\documentclass[12pt]{article}
\usepackage{amsmath, amsfonts, amsthm, amssymb, epsfig}
\newtheorem{theorem}{Theorem}[section]
\newtheorem{lemma}{Lemma}[section]

\numberwithin{equation}{section}
\begin{document}
\newcommand{\ra}{\rightarrow}

\newcommand{\la}{\leftarrow}

\newcommand{\Lra}{\Leftrightarrow}

\newcommand{\Ra}{\Rightarrow}

\newcommand{\BQ}{\Bbb Q}

\newcommand{\BZ}{\Bbb Z}

\newcommand{\BC}{\Bbb C}

\newcommand{\BR}{\Bbb R}

\newcommand{\BN}{\Bbb N}

\newcommand{\df}{\displaystyle \frac}

\newcommand{\lang}{\langle}

\newcommand{\rang}{\rangle}

\newcommand{\ions}{\rm{\scriptstyle ions}}

\newcommand{\ion}{\rm{\scriptstyle ion}}

\newcommand{\dip}{\rm{\scriptstyle dip}}

\newcommand{\pol}{\rm{\scriptstyle pol}}

\newcommand{\mg}{\rm{\scriptstyle mag}}

\newcommand{\free}{\rm{\scriptstyle free}}

\newcommand{\be}{\begin{equation}}

\newcommand{\ee}{\end{equation}}

\newcommand{\bma}{\mbox{\boldmath{$a$}}}

\newcommand{\bmb}{\mbox{\boldmath{$b$}}}

\newcommand{\bmc}{\mbox{\boldmath{$c$}}}

\newcommand{\bmd}{\mbox{\boldmath{$d$}}}

\newcommand{\bme}{\mbox{\boldmath{$e$}}}

\newcommand{\bmf}{\mbox{\boldmath{$f$}}}

\newcommand{\bmg}{\mbox{\boldmath{$g$}}}

\newcommand{\bmh}{\mbox{\boldmath{$h$}}}

\newcommand{\bmi}{\mbox{\boldmath{$i$}}}

\newcommand{\bmj}{\mbox{\boldmath{$j$}}}

\newcommand{\bmk}{\mbox{\boldmath{$k$}}}

\newcommand{\bml}{\mbox{\boldmath{$l$}}}

\newcommand{\bmm}{\mbox{\boldmath{$m$}}}

\newcommand{\bmn}{\mbox{\boldmath{$n$}}}

\newcommand{\bmo}{\mbox{\boldmath{$o$}}}

\newcommand{\bmp}{\mbox{\boldmath{$p$}}}

\newcommand{\bmq}{\mbox{\boldmath{$q$}}}

\newcommand{\bmr}{\mbox{\boldmath{$r$}}}

\newcommand{\bms}{\mbox{\boldmath{$s$}}}

\newcommand{\bmt}{\mbox{\boldmath{$t$}}}

\newcommand{\bmu}{\mbox{\boldmath{$u$}}}

\newcommand{\bmv}{\mbox{\boldmath{$v$}}}

\newcommand{\bmw}{\mbox{\boldmath{$w$}}}

\newcommand{\bmx}{\mbox{\boldmath{$x$}}}

\newcommand{\bmy}{\mbox{\boldmath{$y$}}}

\newcommand{\bmz}{\mbox{\boldmath{$z$}}}

\newcommand{\bmA}{\mbox{\boldmath{$A$}}}

\newcommand{\bmB}{\mbox{\boldmath{$B$}}}

\newcommand{\bmC}{\mbox{\boldmath{$C$}}}

\newcommand{\bmD}{\mbox{\boldmath{$D$}}}

\newcommand{\bmE}{\mbox{\boldmath{$E$}}}

\newcommand{\bmF}{\mbox{\boldmath{$F$}}}

\newcommand{\bmG}{\mbox{\boldmath{$G$}}}

\newcommand{\bmH}{\mbox{\boldmath{$H$}}}

\newcommand{\bmI}{\mbox{\boldmath{$I$}}}

\newcommand{\bmJ}{\mbox{\boldmath{$J$}}}

\newcommand{\bmK}{\mbox{\boldmath{$K$}}}

\newcommand{\bmL}{\mbox{\boldmath{$L$}}}

\newcommand{\bmM}{\mbox{\boldmath{$M$}}}

\newcommand{\bmN}{\mbox{\boldmath{$N$}}}

\newcommand{\bmO}{\mbox{\boldmath{$O$}}}

\newcommand{\bmP}{\mbox{\boldmath{$P$}}}

\newcommand{\bmQ}{\mbox{\boldmath{$Q$}}}

\newcommand{\bmR}{\mbox{\boldmath{$R$}}}

\newcommand{\bmS}{\mbox{\boldmath{$S$}}}

\newcommand{\bmT}{\mbox{\boldmath{$T$}}}

\newcommand{\bmU}{\mbox{\boldmath{$U$}}}

\newcommand{\bmV}{\mbox{\boldmath{$V$}}}

\newcommand{\bmW}{\mbox{\boldmath{$W$}}}

\newcommand{\bmX}{\mbox{\boldmath{$X$}}}

\newcommand{\bmY}{\mbox{\boldmath{$Y$}}}

\newcommand{\bmZ}{\mbox{\boldmath{$Z$}}}

\newcommand{\bmchi}{\mbox{\boldmath{$\chi$}}}

\newcommand{\pp}{\prime\prime}

\newcommand{\bc}{\begin{center}}

\newcommand{\ec}{\end{center}}

\newcommand{\bmcalM}{\mbox{\boldmath{${\cal M}$}}}

\newcommand{\ppp}{\prime\prime\prime}

\newcommand{\aup}{\uparrow}

\newcommand{\adn}{\downarrow}

\newcommand{\sbma}{\mbox{\scriptsize{\boldmath{$a$}}}}

\newcommand{\sbmb}{\mbox{\scriptsize{\boldmath{$b$}}}}

\newcommand{\sbmc}{\mbox{\scriptsize{\boldmath{$c$}}}}

\newcommand{\sbmd}{\mbox{\scriptsize{\boldmath{$d$}}}}

\newcommand{\sbme}{\mbox{\scriptsize{\boldmath{$e$}}}}

\newcommand{\sbmf}{\mbox{\scriptsize{\boldmath{$f$}}}}

\newcommand{\sbmg}{\mbox{\scriptsize{\boldmath{$g$}}}}

\newcommand{\sbmh}{\mbox{\scriptsize{\boldmath{$h$}}}}

\newcommand{\sbmi}{\mbox{\scriptsize{\boldmath{$i$}}}}

\newcommand{\sbmj}{\mbox{\scriptsize{\boldmath{$j$}}}}

\newcommand{\sbmk}{\mbox{\scriptsize{\boldmath{$k$}}}}

\newcommand{\sbml}{\mbox{\scriptsize{\boldmath{$l$}}}}

\newcommand{\sbmm}{\mbox{\scriptsize{\boldmath{$m$}}}}

\newcommand{\sbmn}{\mbox{\scriptsize{\boldmath{$n$}}}}

\newcommand{\sbmo}{\mbox{\scriptsize{\boldmath{$o$}}}}

\newcommand{\sbmp}{\mbox{\scriptsize{\boldmath{$p$}}}}

\newcommand{\sbmq}{\mbox{\scriptsize{\boldmath{$q$}}}}

\newcommand{\sbmr}{\mbox{\scriptsize{\boldmath{$r$}}}}

\newcommand{\sbms}{\mbox{\scriptsize{\boldmath{$s$}}}}

\newcommand{\sbmt}{\mbox{\scriptsize{\boldmath{$t$}}}}

\newcommand{\sbmu}{\mbox{\scriptsize{\boldmath{$u$}}}}

\newcommand{\sbmv}{\mbox{\scriptsize{\boldmath{$v$}}}}

\newcommand{\sbmx}{\mbox{\scriptsize{\boldmath{$x$}}}}

\newcommand{\sbmy}{\mbox{\scriptsize{\boldmath{$y$}}}}

\newcommand{\sbmz}{\mbox{\scriptsize{\boldmath{$z$}}}}

\newcommand{\sbmA}{\mbox{\scriptsize{\boldmath{$A$}}}}

\newcommand{\sbmB}{\mbox{\scriptsize{\boldmath{$B$}}}}

\newcommand{\sbmC}{\mbox{\scriptsize{\boldmath{$C$}}}}

\newcommand{\sbmD}{\mbox{\scriptsize{\boldmath{$D$}}}}

\newcommand{\sbmE}{\mbox{\scriptsize{\boldmath{$E$}}}}

\newcommand{\sbmF}{\mbox{\scriptsize{\boldmath{$F$}}}}

\newcommand{\sbmG}{\mbox{\scriptsize{\boldmath{$G$}}}}

\newcommand{\sbmH}{\mbox{\scriptsize{\boldmath{$H$}}}}

\newcommand{\sbmI}{\mbox{\scriptsize{\boldmath{$I$}}}}

\newcommand{\sbmJ}{\mbox{\scriptsize{\boldmath{$J$}}}}

\newcommand{\sbmK}{\mbox{\scriptsize{\boldmath{$K$}}}}

\newcommand{\sbmL}{\mbox{\scriptsize{\boldmath{$L$}}}}

\newcommand{\sbmM}{\mbox{\scriptsize{\boldmath{$M$}}}}

\newcommand{\sbmN}{\mbox{\scriptsize{\boldmath{$N$}}}}

\newcommand{\sbmO}{\mbox{\scriptsize{\boldmath{$O$}}}}

\newcommand{\sbmP}{\mbox{\scriptsize{\boldmath{$P$}}}}

\newcommand{\sbmQ}{\mbox{\scriptsize{\boldmath{$Q$}}}}

\newcommand{\sbmR}{\mbox{\scriptsize{\boldmath{$R$}}}}

\newcommand{\sbmS}{\mbox{\scriptsize{\boldmath{$S$}}}}

\newcommand{\sbmT}{\mbox{\scriptsize{\boldmath{$T$}}}}

\newcommand{\sbmU}{\mbox{\scriptsize{\boldmath{$U$}}}}

\newcommand{\sbmV}{\mbox{\scriptsize{\boldmath{$V$}}}}

\newcommand{\sbmW}{\mbox{\scriptsize{\boldmath{$W$}}}}

\newcommand{\sbmX}{\mbox{\scriptsize{\boldmath{$X$}}}}

\newcommand{\sbmY}{\mbox{\scriptsize{\boldmath{$Y$}}}}

\newcommand{\sbmZ}{\mbox{\scriptsize{\boldmath{$Z$}}}}

\newcommand{\bmnabla}{\mbox{\boldmath{$\nabla$}}}

\newcommand{\bmalpha}{\mbox{\boldmath{$\alpha$}}}

\newcommand{\bmone}{\mbox{\boldmath{$1$}}}

\newcommand{\bmsigma}{\mbox{\boldmath{$\sigma$}}}

\newcommand{\bmmu}{\mbox{\boldmath{$\mu$}}}

\newcommand{\bmell}{\mbox{\boldmath{$\ell$}}}

\newcommand{\bmGamma}{\mbox{\boldmath{$\Gamma$}}}

\newcommand{\La}{Leftarrow}

\newcommand{\raAB}{\overset{\ra}{AB}}

\newcommand{\raAA}{\overset{\ra}{AA}}

\newcommand{\lrda}{\longleftrightarrow \hspace*{-.25in}\longleftrightarrow}

 \newcommand{\lrra}{\longrightarrow \hspace*{-.21in}\longrightarrow}

\newcommand{\llla}{\longleftarrow \hspace*{-.21in}\longleftarrow}

\newcommand{\betallla}{\underset{\hspace*{-.25in}\beta}{\llla}}

\newcommand{\betaYlrra}{\underset{\hspace*{.2in}\beta Y}{\lrra}}

\newcommand{\betaYllla}{\underset{\hspace*{-.2in}\beta Y}{\llla}}

\newcommand{\lbetaetaYlrda}{\underset{\hspace*{-.2in}\beta \eta Y}{\lrda}}

\newcommand{\rbetaetaYlrda}{\underset{\hspace*{.2in}\beta \eta Y}{\lrda}}

\newcommand{\cbetaetaYlrda}{\underset{\beta \eta Y}{\lrda}}

\newcommand{\Ylrra}{\underset{\hspace*{.25in}Y}{\lrra}}

\newcommand{\rbetaetalrda}{\underset{\hspace*{.25in}\beta\eta}{\lrda}}

\newcommand{\rbetaYlrda}{\underset{\hspace*{.25in}\beta Y}{\lrda}}

\newcommand{\lbetaYlrda}{\underset{\hspace*{-.25in}\beta Y}{\lrda}}

\newcommand{\cbetaetalrda}{\underset{\beta\eta}{\lrda}}

\newcommand{\lbetaetalrda}{\underset{\hspace*{-.25in}\beta\eta}{\lrda}}

\newcommand{\Yllla}{\underset{\hspace*{-.25in}Y}{\llla}}

\newcommand{\betaetallla}{\underset{\hspace*{-.25in}\beta\eta}{\llla}}

\newcommand{\betaetalrra}{\underset{\hspace*{.25in}\beta\eta}{\lrra}}

\newcommand{\betaetaYlrra}{\underset{\hspace*{.15in}\beta \eta Y}{\lrra}}

\newcommand{\betaetaYllla}{\underset{\hspace*{-.15in}\beta \eta Y}{\llla}}

\newcommand{\betaetaYnlrda}{\underset{\!\!\!\!\!\beta\eta Y}{\not \!\!\!\! \! \lrda}}

\newcommand{\2}{\underline{2}}

\newcommand{\n}{\underline{n}}

\newcommand{\rat}{\rightarrowtail\!\!\!\!\rightarrow}

\newcommand{\lat}{\leftarrow\!\!\!\!\leftarrowtail}

\newcommand{\rra}{\rightarrow\!\!\!\!\rightarrow}

\newcommand{\uk}{\underline{k}}

\newcommand{\m}{\underline{m}}

\newcommand{\0}{\underline{0}}

\newcommand{\1}{\underline{1}}

\newcommand{\zk}{\underline{k-1}}

\newcommand{\xk}{\underline{k+1}}

\newcommand{\yP}{\underline{P}}

\newcommand{\boxk}{\framebox{\uk}}

\newcommand{\fzk}{\framebox{\zk}}

\newcommand{\fxk}{\framebox{xk}}

\newcommand{\boxm}{\framebox{\m}}

\newcommand{\boxn}{\framebox{\n}}

\newcommand{\rhu}{\rightharpoonup}

\newcommand{\orhua}{\overset{\rhu}{a}}

\newcommand{\orhub}{\overset{\rhu}{b}}

\newcommand{\orhui}{\overset{\rhu}{\imath}}

\newcommand{\orhuj}{\overset{\rhu}{\jmath}}

\newcommand{\orhuk}{\overset{\rhu}{k}}

\newcommand{\orhuF}{\overset{\rhu}{F}}

\newcommand{\orhuG}{\overset{\rhu}{G}}

\newcommand{\orhugj}{\overset{\rhu}{gj}}

\newcommand{\orhuT}{\overset{\rhu}{T}}

\newcommand{\orhuzero}{\overset{\rhu}{0}}

\newcommand{\orhunabla}{\overset{\rhu}{\nabla}}

\newcommand{\orhuu}{\overset{\rhu}{u}}

\newcommand{\orhuPQ}{\overset{\rhu}{PQ}}

\newcommand{\orhuOP}{\overset{\rhu}{OP}}

\newcommand{\orhuPR}{\overset{\rhu}{PR}}

\newcommand{\orhun}{\overset{\rhu}{n}}

\newcommand{\orhuv}{\overset{\rhu}{v}}

\newcommand{\orhur}{\overset{\rhu}{r}}

\newcommand{\orhuN}{\overset{\rhu}{N}}

\newcommand{\orhuS}{\overset{\rhu}{S}}

\newcommand{\orhubeta}{\overset{\rhu}{\beta}}

\newcommand{\orhud}{\overset{\rhu}{d}}

\newcommand{\orhudS}{\overset{\rhu}{dS}}

\newcommand{\oiint}{{\displaystyle

\int\!\!\!\!\!\int\!\!\!\!\!\!\!\!\bigcirc}}

\newcommand{\orhuV}{\overset{\rhu}{V}}

\newcommand{\orhux}{\overset{\rhu}{x}}

\newcommand{\dint}{\displaystyle \int}

\newcommand{\diint}{\displaystyle \iint}

\newcommand{\diiint}{\displaystyle \iiint}

\newcommand{\teta}{\tilde{\eta}}

\newcommand{\ov}{\overline{v}}

\newcommand{\ox}{\overline{x}}

\newcommand{\oX}{\overline{X}}

\newcommand{\oV}{\overline{V}}

\newcommand{\os}{\overline{s}}

\newcommand{\oE}{\overline{E}}

\newcommand{\of}{\overline{f}}

\newcommand{\dsum}{\displaystyle \sum}

\newcommand{\uszG}{\underset{z_1 \in {\cal G}_1}{\dsum}}

\newcommand{\uizG}{\underset{z_i \in {\cal G}_i}{\dsum}}

\newcommand{\uzG}{\underset{z \in {\cal G}}{\dsum}}

\newcommand{\uzGi}{\underset{z\in{\cal G}_i}{\dsum}}

\newcommand{\uj}{\underset{ij}{\dsum}}

\newcommand{\uijk}{\underset{ijk}{\dsum}}

\newcommand{\urhoG}{\underset{\rho_1 \in {\cal G}_1}{\dsum}}

\newcommand{\ui}{\underset{i}{\dsum}}

\newcommand{\obeta}{\overline{\beta}}

\newcommand{\ogamma}{\overline{\gamma}}

\newcommand{\odelta}{\overline{\delta}}

\newcommand{\ovarepsilon}{\overline{\varepsilon}}

\newcommand{\olambda}{\overline{\lambda}}

\newcommand{\uij}{\underset{ij}{\dsum}}

\newcommand{\R}{\mathbb R}

\newcommand{\G}{\mathbb G}

\title{Large Time Behavior of the Relativistic Vlasov Maxwell
System in Low Space Dimension
 }

\author{Robert Glassey\\
Department of Mathematics, Indiana University\\
Bloomington, IN 47405, USA\\
\\
Stephen Pankavich\\
Department of Mathematics, University of Texas at Arlington\\
Arlington, TX 76019, USA\\
\\
Jack Schaeffer\\
Department of Mathematical Sciences, Carnegie Mellon University\\
Pittsburgh, PA 15213, USA}
\date{}
\maketitle

\begin{abstract}
When particle speeds are large the motion of a collisionless
plasma is modeled by the relativistic Vlasov Maxwell system.
Large time behavior of solutions which depend on one position
variable and two momentum variables is considered.  In the case of
a single species of charge it is shown that there are solutions
for which the charge density $(\rho = \dint f dv)$ does not
decay in time.  This is in marked contrast to results for the
non-relativistic Vlasov Poisson system in one space dimension.
The case when two oppositely charged species are present and the
net total charge is zero is also considered.  In this case it is
shown that the support in the first component of momentum can grow
at most as $t^{\frac{3}{4}}$.
\end{abstract}

\section{Introduction}

Consider the relativistic Vlasov-Maxwell system:

\be \left\{ \begin{array}{rcl}
\partial_tf^{\alpha} & +& \hat{v}^{\alpha}_1 \partial_x f^{\alpha} +
e^{\alpha}\left(E_1 +\hat{v}^{\alpha}_2 B\right) \partial_{v_1}
f^{\alpha} \\
& + & e^{\alpha} \left(E_2 - \hat{v}^{\alpha}_1
B\right)\partial_{v_2}f^{\alpha} = 0\\
\\
\rho(t,x)&  =& \dint \underset{\alpha}{\dsum}
e^{\alpha}f^{\alpha}(t,x,v) dv\\
\\
j(t,x) &=& \dint \underset{\alpha}{\dsum} e^{\alpha}
f^{\alpha}(t,x,v)\hat{v}^{\alpha} dv\\
\\
E_1(t,x) &=& \dfrac{1}{2} \dint^x_{-\infty} \rho(t,y) dy -
\dfrac{1}{2}\dint^{\infty}_x \rho(t,y) dy\\
\\
\partial_t E_2 &+& \partial_x B = - j_2\\
\\
\partial_t B& +& \partial_x E_2 = 0
\end{array}\right. \label{E1.1}
\ee

\noindent for $\alpha = 1, \ldots , N$.  Here, $t\geq 0$ is time,
$ x\in \BR$ is the first component of position, and $v =
(v_1,v_2)\in \BR^2$ contains the first two components of momentum.
Hence $dv = dv_2 dv_1$ and the $v$ integrals are understood to be
over $\BR^2$.  $f^{\alpha}$ gives the number density in phase
space of particles of mass $m^{\alpha}$ and charge $e^{\alpha}$.
Velocity is given by

$$
\hat{v}^{\alpha} = \dfrac{v}{\sqrt{(m^{\alpha})^2 + |v|^2}}\ ,
$$

\noindent where the speed of light has been normalized to one. The
effects of collisions are neglected.

The initial conditions

$$
\left\{ \begin{array}{rcll} f^{\alpha}(0,x,v) & = & f^{\alpha}_0 (x,v)
\geq 0 & \alpha = 1, \ldots, N\\
\\
E_2(0,x) & = & E_{20}(x)& \\
\\
B(0,x) & = & B_0 (x) &
\end{array}\right.
$$

\noindent are given where it is assumed throughout the paper that
$f^{\alpha}_0 \in C^1_0 (\BR^3)$ is nonnegative and compactly
supported and that $E_{20},B_0 \in C^1_0 (\BR)$ are compactly
supported.  When the neutrality condition,

$$
\diint \underset{\alpha}{\dsum} e^{\alpha}f^{\alpha}_0 dv\, dx = 0,
$$

\noindent holds, we will refer to this as the neutral case.  A
major goal of this paper is to compare the neutral case with the
monocharge case, which may be obtained from (\ref{E1.1}) by
setting $N= 1$.  In the monocharge case we will drop $\alpha$ and
write, for example, $f = f^{\alpha}=f^1$ and take $e^{\alpha}=e^1
= 1,\ m^{\alpha} = m^1 = 1$.

Choose $C_0$ such that $f^{\alpha}_0, E_{20}, B_0$ vanish (for all
$\alpha$) if $|x| \geq C_0$.  The letter $C$ will denote a
positive generic constant which may depend on the initial data
(but not $t,x,v$) and may change from line to line, whereas a
numbered constant (such as $C_0$) has a fixed value.  We also
define the characteristics,
$(X^{\alpha}(s,t,x,v),V^{\alpha}(s,t,x,v))$, of $f^{\alpha}$ by

\be \left\{ \begin{array}{rcll} \dfrac{dX^{\alpha}}{ds} & = &
\hat{V}^{\alpha}_1 & X^{\alpha}(t,t,x,v) = x\\
\\
\dfrac{dV^{\alpha}_1}{ds} & = & e^{\alpha}\left(E_1(x,X^{\alpha})+ \hat{V}^{\alpha}_2
B(s,X^{\alpha})\right)
& V^{\alpha}_1 (t,t,x,v) = v_1\\
\\
\dfrac{dV^{\alpha}_2}{ds} & = & e^{\alpha}\left(E_2 (s,X^{\alpha}) - \hat{V}^{\alpha}_1
B(s,X^{\alpha})\right)& V^{\alpha}_2 (t,t,x,v) = v_2.
\end{array}\right. \label{E1.2}
\ee

\begin{theorem} In the neutral case there is a constant, $C$, such
that

\be
\begin{array}{rcl}
C & \geq & \left. \dint^t_0 \left[E^2_1 + (E_2 - B)^2 + \dint
\underset{\alpha}{\dsum} f^{\alpha} \left( \sqrt{(m^{\alpha})^2 +
|v|^2} - v_1\right) dv \right] \right|_{(\tau, x-t + \tau)}\
d\tau\\
\\
& & \left.+ \dint^t_0 \left[ E^2_1 + (E_2 + B)^2 + \dint
\underset{\alpha}{\dsum} f^{\alpha} \left(\sqrt{(m^{\alpha})^2 +
|v|^2} + v_1\right) dv \right] \right|_{(\tau,x+t - \tau)} \
d\tau
\end{array}\label{E1.3}
\ee

\noindent for all $t \geq 0, x \in \BR$.  In the monocharge case
there is a constant, $C$, such that

\be \left. C(C_0 + t-x) \geq \dint^t_0 \left[ (E_2-B)^2 + f
\dfrac{(\sqrt{1+|v|^2}-v_1)^2}{\sqrt{1+|v|^2}} dv
\right]\right|_{(\tau, x - t + \tau)}\ d\tau \label{E1.4} \ee

\noindent for $x < C_0 + t$ and

\be \left.C(C_0 + t + x) \geq \dint^t_0 \left[ (E_2+B)^2 + f
\dfrac{(\sqrt{1+|v|^2}+v_1)^2}{\sqrt{1+|v|^2}} dv \right]
\right|_{(\tau, x + t - \tau)}\ d\tau \label{E1.5} \ee

\noindent for $-C_0 - t < x$.
\end{theorem}

\noindent The proof of this theorem relies on conservation of energy and of
momentum (in the monocharge case) and is contained in Section 2.

\begin{theorem} In the neutral case there is a constant, $C$, such
that

$$
|v_2| \leq C + C \sqrt{t-|x|+ C_0}
$$

\noindent on the support of $f^{\alpha}$ for every $\alpha$.  In
the monocharge case there is a positive constant, $C$, such that

$$
|v_2| \leq C + C \sqrt{(t+C_0)^2 -x^2}
$$

\noindent on the support of $f$.
\end{theorem}

\noindent The proof of Theorem 1.2 is in Section 3.

\begin{theorem}  There are solutions of the monocharge problem for
which there exist $x_0 \in \BR$ and $C > 0$ such that

\be \dint^{\infty}_{x_0+t} \rho(t,x) dx > C\label{E1.6} \ee

\noindent for all $t \geq 0$.  Furthermore, there exists $C > 0$
such that

$$\| \rho(t,\cdot) \|_{L^p(\BR)} > C
$$

\noindent for all $t \geq 0$ and $p \in [1,\infty]$.
\end{theorem}

\noindent The second assertion of Theorem 1.3 follows from (\ref{E1.6}) by
using H\"{o}lder's inequality:

$$
C < \dint^{C_0+t}_{x_0+t} \rho(t,x)dx \leq \|
\rho(t,\cdot)\|_{L^p(\BR)} (C_0 - x_0)^{1-\frac{1}{p}}.
$$

\noindent The proof of Theorem 1.3 is contained in Section 4.  In \cite{8}
an analogous, but more detailed, result is obtained for the
relativistic Vlasov Poisson system (which may be obtained from
(\ref{E1.1}) by setting $E_2 = B = 0$).

\begin{theorem} For the neutral problem there is a constant, $C$,
such that

\be |v_1| \leq C+Ct^{\frac{1}{2}} (t - |x|+ 2
C_0)^{\frac{1}{4}} \label{E1.7} \ee

\noindent on the support of $f^{\alpha}$ for every $\alpha$.
\end{theorem}

\noindent The proof is in Section 5.  A similar, but different, estimate is
obtained in \cite{8} for the relativistic Vlasov Poisson system.  Also, note that (\ref{E1.7}) rules out an estimate like (\ref{E1.6}).  If (\ref{E1.6}) held, then there would be characteristics for
which $f^{\alpha}  \neq 0$ and

\be X^{\alpha}(t,0,x,v) \geq x_0 + t \label{E1.8} \ee

\noindent for all $t \geq 0$.  Then by (\ref{E1.7}) and
(\ref{E1.8})

$$
\left| V^{\alpha}_1 (t,0,x,v)\right| \leq C+Ct^{\frac{1}{2}}
$$

\noindent so

$$ \begin{array}{rcl} 1-\hat{V}^{\alpha}_1 (t,0,x,v) & = &
\dfrac{1+(V^{\alpha}_2)^2}{\sqrt{1+|V^{\alpha}|^2}(\sqrt{1+|V^{\alpha}|^2} + V^{\alpha}_1)} \\
\\
 & \geq & \dfrac{1}{2(1+|V^{\alpha}|^2)} \geq \dfrac{C}{1+t}
 \end{array}
 $$

 \noindent and

 $$
 \begin{array}{rcl} C\ln (1+t) & \leq & \dint^t_0
 \left(1-\hat{V}^{\alpha}_1(s,0,x,v)\right) ds\\
 \\
 & = & t -X^{\alpha}(t,0,x,v) +x \leq x - x_0
 \end{array}
 $$

 \noindent for all $t \geq 0$.

 Finally Section 6 contains the proof of

 \begin{theorem} In both the neutral and monocharge cases there
 are no nontrivial steady solutions with $f^{\alpha}, E_2$ and $B$ compactly
 supported.
\end{theorem}

The global existence in time of smooth solutions to (\ref{E1.1})
is shown in \cite{9} when a neutralizing background density is
included.  Adaptation of the essential estimate from \cite{9} to
the current situation is briefly discussed in Section 2.  Global
existence has been shown in two dimensions, \cite{11}, and two and
one-half dimensions, \cite{10}, but is open for large data in three
dimensions. Some time decay is known for the classical Vlasov
Poisson system in three dimensions (\cite{12}, \cite{14},
\cite{15}). Additionally there are time decay results for the
classical Vlasov Poisson system in one dimension (\cite{1},
\cite{2}, \cite{7}, \cite{17}). For decay results on the
relativistic  Vlasov Poisson system, see \cite{7} and \cite{13}.
References \cite{3}, \cite{4}, and \cite{5} are also mentioned
since they deal with time dependent rescalings and time decay for
other kinetic equations.  We also cite \cite{6} and \cite{16} as
general references on mathematical kinetic theory.

A main point to this article is that the non-decay stated in
Theorem 1.3 is in marked contrast to the decay found in \cite{1},
\cite{2}, and \cite{17}.  In \cite{1}, \cite{2} and \cite{17} the
problem studied is non-relativistic.  Hence, there is no apriori
upper bound on particle speed and this leads to dispersion.  In
this paper (and also \cite{8}) particle speeds are bounded by the
speed of light and this limits the dispersion.

\section{Conservation Laws}

Define

$$
\begin{array}{rcl}
e & = & \dint \underset{\alpha}{\dsum} f^{\alpha}
\sqrt{(m^{\alpha})^2 + |v|^2} dv + \dfrac{1}{2} |E|^2 +
\dfrac{1}{2} B^2,\\
\\
m & = & \dint \underset{\alpha}{\dsum} f^{\alpha}v_1 dv + E_2B,
\end{array}
$$

\noindent and

$$
\ell = \dint \underset{\alpha}{\dsum} f^{\alpha} v_1
\hat{v}^{\alpha}_1 dv - \dfrac{1}{2} E^2_1 + \dfrac{1}{2}E^2_2 +
\dfrac{1}{2} B^2.
$$

\noindent A short computation reveals that

\be \partial_t e + \partial_x m = 0 \label{E2.1} \ee

\noindent and

\be \partial_t m + \partial_x \ell = 0. \label{E2.2} \ee

Using (\ref{E2.1}), the divergence theorem yields

\be \begin{array}{rcl} 0 & = & \dint^t_0
\dint^{x+t-\tau}_{x-t+\tau} \left(\partial_{\tau} e + \partial_y m
\right) dy\ d\tau\\
\\
& = & \left. \dint^t_0 (e + m) \right|_{(\tau,x+t-\tau)}  d\tau +
\left. \dint^t_0 (e-m) \right|_{(\tau,x-t+\tau)} \ d \tau\\
\\
&& - \dint ^{x+t}_{x-t} e(0,y) dy.
\end{array} \label{E2.3}
\ee

\noindent Note that

$$
\begin{array}{rcl}
e\pm m & = & \dint \underset{\alpha}{\dsum} f^{\alpha}
\left(\sqrt{(m^{\alpha})^2 + |v|^2} \pm v_1\right) dv +
\dfrac{1}{2} E^2_1 \\
\\
&& + \dfrac{1}{2} (E_2 \pm B)^2 \geq 0
\end{array}
$$

\noindent and that

$$
\begin{array}{rcl}
|j_2| & \leq & \dint \underset{\alpha}{\dsum} f^{\alpha}
\dfrac{|v_2|}{\sqrt{(m^{\alpha})^2 + |v|^2}} \ dv\\
\\
& \leq & C\dint \underset{\alpha}{\dsum} f^{\alpha} \left(
\sqrt{(m^{\alpha})^2+|v|^2} \pm v_1\right)\ dv.
\end{array}
$$

\noindent In the neutral case (\ref{E2.3}) yields

\be
\begin{array}{rcl}
C & \geq &\left. \dint^{x+t}_{x-t} e(0,y)dy = \dint^t_0
(e+m)\right|_{(\tau,x+t-\tau)}\ d\tau\\
\\
&& \left. + \dint^t_0 (e-m)\right|_{(\tau,x-t+\tau)}\ d\tau.
\end{array} \label{E2.4}
\ee

\noindent In the monocharge case, since $E_1$ is not compactly
supported, (\ref{E2.3}) only yields

$$
\begin{array}{rcl}
Ct & \geq & \left.\dint^t_0 (e+m)\right|_{(\tau, x+t-\tau)}\ d
\tau \\
\\
&& \left.+ \dint^t_0 (e-m) \right|_{(\tau,x-t+\tau)} \ d\tau.
\end{array}
$$

\noindent It follows that

\be\begin{array}{rcl} |E_2| + |B| & \leq & C+C\dint^t_0 \left| j_2
(\tau, x+t-\tau)\right| d\tau\\
\\
&& \ \ + C\dint^t_0 \left| j_2 (\tau, x-t+\tau)\right| d\tau\\
\\
& \leq & C + Ct^p
\end{array} \label{E2.5}
\ee

\noindent where $p=0$ in the neutral case and $p=1$ in the
monocharge case.  Global existence of smooth solutions follows in
both cases as in \cite{9}.

Consider the monocharge case now.  Bounds independent of $t$ may
be obtained by also using (\ref{E2.2}).  For $x_0 < C_0$ the
divergence theorem yields

$$\begin{array}{rcl}
0 & = & \dint^t_0 \dint^{C_0+\tau}_{x_0 + \tau} \left[
\partial_{\tau} (e-m) + \partial_y (m-\ell) \right] dy\ d\tau\\
\\
& = & \dint^t_0 \left[ e-2m + \ell\right]\left|_{(\tau,x_0+\tau)}
d\tau + \dint^{C_0+t}_{x_0+t} (e-m)\right|_{(t,y)}  dy \\
\\
&& - \dint^t_0 [e-2m+ \ell]\left|_{(\tau,C_0 + \tau)} d\tau -
\dint^{C_0}_{x_0} (e-m) \right|_{(0,y)} dy.
\end{array}
$$

\noindent Note that

$$
e-2 m + \ell = \dint \underset{\alpha}{\dsum} f^{\alpha}
\dfrac{(\sqrt{(m^{\alpha})^2 + |v|^2}-v_1)^2}{\sqrt{(m^{\alpha})^2
+ |v|^2}} dv + (E_2-B)^2
$$

\noindent is nonnegative and vanishes on $y = C_0 + \tau$ (since
$E_1$ canceled).  Hence,

\be \begin{array}{rcl} C(C_0-x_0) & \geq & \left.\dint^{C_0}_{x_0}
(e-m)\right|_{(0,y)} dy \\
\\
& = & \left.\dint^t_0 (e-2m+\ell)\right|_{(\tau,x_0+\tau)} d\tau +
\left.\dint^{C_0 + t}_{x_0+t} (e-m) \right|_{(t,y)} dy.
\end{array}
\label{E2.6} \ee

\noindent Similarly

$$
e+2m+\ell = \dint \underset{\alpha}{\dsum} f^{\alpha}
\dfrac{(\sqrt{(m^{\alpha})^2 +|v|^2}+v_1)^2}{\sqrt{(m^{\alpha})^2
+ |v|^2}} dv + (E_2 + B)^2
$$

\noindent is nonnegative and vanishes on $y = -C_0 - \tau$ and

$$
0 = \dint^t_0 \dint^{x_0-\tau}_{-C_0 - \tau} \left[
\partial_{\tau} (e+m) + \partial_y(m+\ell)\right] dy\ d\tau
$$

\noindent leads to

\be \begin{array}{rcl} C(x_0+C_0) & \geq &
\left.\dint^{x_0}_{-C_0} (e+m) \right|_{(0,y)} dy \\
\\
& = &\left. \dint^t_0 (e+2m + \ell)\right|_{(\tau,x_0 - \tau)}
d\tau + \left.\dint^{x_0-t}_{-C_0 -t} (e+m)\right|_{(t,y)} dy
\end{array}
\label{E2.7} \ee

\noindent for $x_0 > -C_0$.  Theorem 1.1 now follows from
(\ref{E2.4}), (\ref{E2.6}), and (\ref{E2.7}).

\section{Bounds on $v_2$ Support}

Define

$$
A(t,x) = \dint^x_{-\infty} B(t,y)dy
$$

\noindent and note that

$$\begin{array}{rcl}\partial_tA+\partial_x A & = & -(E_2-B),\\
\\
\partial_t A - \partial_x A & = & - (E_2+B),
\end{array}
$$

\noindent so

\be \begin{array}{rcl} A(t,x) & = & \left. A(0,x-t) -
\dint^t_0(E_2-B)\right|_{(\tau, x-t+\tau)} d\tau\\
\\
& = & \left.A(0,x+t) - \dint^t_0 (E_2+B)\right|_{(\tau, x + t -
\tau)} d\tau.\end{array} \label{E3.1} \ee

\noindent For $|x| \geq C_0 + t, |A(t,x)| = |A(0,x-t)| \leq C$, so
consider $ |x| < C_0 + t$.  Then (\ref{E3.1}) becomes

$$
\begin{array}{rcl}  A(t,x) & = & \left.A(0,x-t) -
\dint^t_{\max(0,\frac{t-x-C_0}{2})} (E_2-B) \right|_{(\tau, x-t+
\tau)} d\tau\\
\\
& = & \left.A(0,x+t) - \dint^t_{\max(0,\frac{x+t-C_0}{2})} (E_2 +
B) \right|_{(\tau, x+t - \tau)}  d\tau. \end{array}
$$

\noindent In the neutral case, (\ref{E1.3}) and the Cauchy
Schwartz inequality yield

$$\begin{array}{rcl}
|A(t,x)| & \leq & C + \sqrt{t-\max(0,\frac{t-x-C_0}{2})}
\sqrt{C}\\
\\
& \leq & C + C \sqrt{t+x+C_0}
\end{array}
$$

\noindent and

$$
\begin{array}{rcl}
|A(t,x)| & \leq & C + \sqrt{t-\max (0,\frac{x+t-C_0}{2})}
\sqrt{C}\\
\\
& \leq & C + C \sqrt{t-x+C_0}.
\end{array}
$$

\noindent Hence

\be |A(t,x)| \leq C + C \sqrt{t-|x|+C_0}\label{E3.2} \ee

\noindent follows.  For the monocharge case, (\ref{E1.4}) is used
in place of (\ref{E1.3}) to obtain

\be \begin{array}{rcl} |A(t,x)| & \leq & C + \sqrt{t-\max
(0,\frac{t-x-C_0}{2})} \sqrt{C(C_0 + t - x)}\\
\\
 & \leq & C + C \sqrt{(t+C_0)^2 - x^2}.
 \end{array}\label{E3.3}
 \ee

 From (\ref{E1.2}) we have

 $$
 f^{\alpha}\left(s,X^{\alpha}(s,t,x,v),V^{\alpha}(s,t,x,v)\right)
 = f^{\alpha}(t,x,v)
 $$

 \noindent and

 $$
 V^{\alpha}_2 (s,t,x,v) + e^{\alpha}A(s,X^{\alpha}(s,t,x,v)) = v_2
 + e^{\alpha}A(t,x)
 $$

 \noindent for all $s,t,x,v$.  If $f^{\alpha}(t,x,v)\neq 0$ then

 $$
\begin{array}{rcl} \left|v_2 + e^{\alpha} A(t,x)\right| & = &
\left|V^{\alpha}_2(0,t,x,v) +e^{\alpha}A(0,X^{\alpha}(0,t,x,v)) \right|\\
\\
& \leq & C.
\end{array}
$$

\noindent In the neutral case, (\ref{E3.2}) yields

\be |v_2| \leq C+C \sqrt{t-|x|+C_0}. \label{E3.4} \ee

\noindent In the monocharge case, (\ref{E3.3}) yields

\be |v_2| \leq C+C \sqrt{(t+C_0)^2 - x^2} \label{E3.5} \ee

\noindent on the support of $f^{\alpha}$.

Theorem 1.2 follows from (\ref{E3.4}) and (\ref{E3.5}) but we make
one further observation.  On the support of $f^{\alpha}$

$$
|v_2 +e^{\alpha}A(t,x)| \leq C
$$

\noindent so $v_2 \in (-e^{\alpha}A(t,x)-C,-e^{\alpha}A(t,x)+C)$.
Thus the $v_2$ support has bounded measure.

\section{Non-decay of $\rho$ in the Monocharge Case}

In this section only the monocharge case is considered.  Let

$$M = \dint \rho(t,x)dx
$$

\noindent and note that

$$
E_1 = \dfrac{1}{2} \dint^x_{-\infty} \rho \, dy -
\dfrac{1}{2}\dint^{\infty}_x \rho\, dy = \dfrac{1}{2} M -
\dint^{C_0+t}_x \rho\, dy.
$$

\noindent For some $x_0 \in (-C_0, C_0)$ define

$$
\mu(t) = \dint^{C_0+t}_{x_0+t} \rho(t,y)dy
$$

\noindent and

$$
\left. {\cal E}(t) = \dint^{C_0+t}_{x_0+t} (e-m)\right|_{(t,y)}
dy.
$$

\noindent Then

$$\mu^{\prime}(t)  = j_1(t,x_0 + t) - \rho (t,x_0+t) \leq 0
$$

\noindent and by (\ref{E2.1}) and (\ref{E2.2})

$$
\begin{array}{rcl}
{\cal E}^{\prime}(t) & = &\left. \left.  -\dint^{C_0+t}_{x_0+t}
\partial_y (m-\ell)dy + (e-m)\right|_{(t,C_0+t)} -
(e-m)\right|_{(t,x_0 + t)}\\
\\
&=& (e-2m+\ell)\left|_{(t,C_0+t)} - (e-2m + \ell)
\right|_{(t,x_0+t)}\\
\\
& = & \left.-(e-2m+\ell)\right|_{(t,x_0+t)} \leq 0.
\end{array}
$$

\noindent Suppose that

\be \dfrac{1}{2}M \geq \mu(0) \label{E4.1} \ee

\noindent and

\be \dfrac{1}{2} (C_0 - x_0) (\dfrac{1}{2}M)^2 > {\cal E}(0).
\label{E4.2} \ee

\noindent Then for $y \geq x_0 + t$,

$$
\begin{array}{rcl}
E_1(t,y) \geq E_1(t,x_0 + t) & = & \dfrac{1}{2}M - \mu (t) \\
\\
& \geq & \dfrac{1}{2}M - \mu(0) \geq 0
\end{array}
$$

\noindent so

$$
\begin{array}{rcl}
{\cal E}(0) & \geq & {\cal E}(t) \geq
\dfrac{1}{2}\dint^{C_0+t}_{x_0+t} E^2_1 dy\\
\\
& \geq & \dfrac{1}{2} (C_0 - x_0) ( \dfrac{1}{2}M - \mu(t))^2
\end{array}
$$

\noindent and hence

$$
\sqrt{\dfrac{2{\cal E}(0)}{C_0-x_0}} \geq \dfrac{1}{2} M - \mu (t)
$$

\noindent and

$$
\dint^{C_0+t}_{x_0+t}  \rho\, dy = \mu (t) \geq  \dfrac{1}{2} M -
\sqrt{\dfrac{2{\cal E}(0)}{C_0-x_0}} > 0.
$$

\noindent Hence Theorem 1.3 follows from (\ref{E4.1}) and
(\ref{E4.2}).

To see that there are initial conditions for which (\ref{E4.1})
and (\ref{E4.2})  hold consider the following:  Let $f^L_0, f^R_0
\in C^1_0 (\BR^3)$ be nonnegative and compactly supported with

$$
\begin{array}{lcl}  f^{L}_0 (x,v) =0 & {\rm if} & x \geq -1,\\
\\
f^R_0 (x,v) = 0 & {\rm if} & x \notin (-1,0),
\end{array}
$$

\noindent and

$$
\dfrac{1}{2} \diint f^L_0 dv\, dx \geq \diint f^R_0 dv\, dx > 0.
$$

\noindent Let

$$C_0 = \sup \left\{ |x| : f^L_0 (x,v) \neq 0 \ {\rm for\ some}\
v\right\}
$$

\noindent and

$$
f(0,x,v) = f^L_0 (x,v) + f^R_0 (x-C_0,v_1-W,v_2)
$$

\noindent for $W>1$.  Taking $x_0 = -1$ we have

$$
\mu(0) = \diint f^R_0 dv\, dx \leq \dfrac{1}{2}M,
$$

\noindent which is (\ref{E4.1}).  Taking

$$
E_2(0,y) = B(0,y) = 0$$

\noindent (and using $x_0 = -1$) we have

$$\begin{array}{rcl}
{\cal E}(0) & = & \left.\dint^{C_0}_{x_0} \left[ \dint f (
\sqrt{1+|v|^2}-v_1) dv + \dfrac{1}{2}E^2_1\right]\right|_{(0,y)}
dy \\
\\
& = & \dint^{C_0}_{x_0} \left[ \dint f^R_0 (y-C_0, v_1 - W,v_2)
\dfrac{1+v^2_2}{\sqrt{1+|v|^2} + v_1} dv + \dfrac{1}{2}E^2_1
\right] dy \\
\\
& \leq & \dfrac{C}{W} + \dfrac{1}{2} \dint^{C_0-1}_{x_0} \left(
\dfrac{1}{2}M  - \mu(0)\right)^2 dy + \dfrac{1}{2}
\dint^{C_0}_{C_0-1} (\dfrac{1}{2}M)^2 dy \\
\\
& = & \dfrac{C}{W} + \dfrac{C_0 -1-x_0}{2} \left( \dfrac{M^2}{4} -
M \mu(0) + \mu^2(0) \right) + \dfrac{1}{8} M^2 \\
\\
& = & \dfrac{C}{W} + \dfrac{C_0}{2} \left(\dfrac{M^2}{4} - M\mu(0)
+ \mu^2(0)\right) - \dfrac{x_0}{8} M^2\\
\\
& = & \dfrac{C}{W} + \dfrac{C_0-x_0}{8} M^2 - \dfrac{C_0}{2}
\mu(0) (M-\mu(0)).
\end{array}
$$

\noindent Now taking $W$ sufficiently large yields (\ref{E4.2})
completing the proof.

\section{Bounds on $v_1$ Support in the Neutral Case}

In this section we consider only the neutral case.  Define

$$k= \dint \underset{\alpha}{\dsum}
f^{\alpha}\sqrt{(m^{\alpha})^2+|v|^2} dv
$$

\noindent and

$$
\sigma _{\pm}  = \dint \underset{\alpha}{\dsum} f^{\alpha} \left(
\sqrt{(m^{\alpha})^2 + |v|^2}\pm v_1\right) dv.
$$

\noindent Then (\ref{E2.1}) yields

$$
\dint kdx\leq \dint edx = \dint e(0,x)dx = C.$$

\noindent Also (\ref{E1.3}) yields

$$
\dint^t_0 \left[ \sigma _- (\tau,x-t + \tau) + \sigma_+
(\tau,x+t-\tau)\right]  d\tau \leq C.
$$

\noindent These bounds are used in the following:

\begin{lemma} For all $t \geq 0$ and $x \in \BR$

\be \dint \underset{\alpha}{\dsum}
f^{\alpha} dv \leq C\sqrt{k\sigma_-} \label{E5.1} \ee

\noindent and

\be \dint \underset{\alpha}{\dsum}
f^{\alpha} dv \leq C\sqrt{k\sigma_+}.\label{E5.2} \ee
\end{lemma}

\begin{proof} We will show (\ref{E5.1}), the proof of (\ref{E5.2})
is similar.  For any $R\geq 0$

$$
\dint \underset{\alpha}{\dsum} f^{\alpha} dv \leq
\underset{|v|\leq R}{\dint} \underset{\alpha}{\dsum} f^{\alpha} \
dv + \dfrac{Ck}{\sqrt{1+R^2}}.
$$

\noindent For $|v| \leq R$,

$$
\begin{array}{rcl}
\sqrt{(m^{\alpha})^2+|v|^2}-v_1 &=& \dfrac{(m^{\alpha})^2 +
v^2_2}{\sqrt{(m^{\alpha})^2 + |v|^2} + v_1} \geq
\dfrac{(m^{\alpha})^2}{2\sqrt{(m^{\alpha})^2 +|v|^2}}\\
\\
& \geq & \dfrac{C}{\sqrt{1+R^2}}
\end{array}
$$

\noindent so

$$
\begin{array}{rcl}
\dint \underset{\alpha}{\dsum}  f^{\alpha} dv & \leq &
\underset{|v|\leq R}{\dint} \underset{\alpha}{\dsum} f^{\alpha}
C\sqrt{1+R^2} \left( \sqrt{(m^{\alpha})^2+|v|^2} - v_1 \right) dv+ \dfrac{Ck}{\sqrt{1+R^2}}
\\
 & \leq & C\sqrt{1+R^2}\sigma_- +
\dfrac{Ck}{\sqrt{1+R^2}}.
\end{array}
$$

\noindent If $0 < \sigma_- \leq k$, taking

$$
R= \sqrt{\frac{k}{\sigma_-} -1}
$$

\noindent leads to (\ref{E5.1}).

If $k < \sigma_-$ then

$$
\dint \underset{\alpha}{\dsum}  f^{\alpha} dv < Ck < C\sqrt{k\sigma_-}
$$

\noindent and if $\sigma_- = 0$ then

$$
\dint \underset{\alpha}{\dsum}  f^{\alpha} dv = \sqrt{k\sigma_-} = 0.
$$

\noindent In all cases (\ref{E5.1}) holds so the proof is
complete.
\end{proof}

Consider a characteristic

$$
(X(s),V(s)) = (X^{\alpha}(s,0,\ox,\ov),V^{\alpha}(s,0,\ox,\ov))
$$

\noindent of $f^{\alpha}$ (defined in (\ref{E1.2})) along which
$f^{\alpha}(s,X(s),V(s)) \neq 0$.
The idea to the following estimate is that as long as $V_1$ is
large, the integration

$$
\left. V_1(t) =V_1(t-\Delta) + \dint^t_{t-\Delta} e^{\alpha}
\left(E_1 + \hat{V}_2(s)B\right)\right|_{(s,X(s))}\, ds
$$

\noindent is nearly integration on a light cone and (\ref{E1.3})
can be used to obtain an improved estimate.

Define
$$
C_1 = \sup \left\{ |v_1| : \exists t \in [0,1], \ x \in \BR, v_2
\in \BR \ {\rm with}\ \underset{\alpha}{\dsum} f^{\alpha} (t,x,v) \neq 0\right\}\\
$$

\noindent and suppose that $t>0$ and

$$V_1(t) > 2C_1.
$$

\noindent Define

$$
\Delta = \sup \left\{ \tau \in (0,t]:V_1(s) \geq \dfrac{1}{2} V_1
(t)\ {\rm for \ all}\ s \in [t-\tau,t]\right\}.
$$

\noindent Note that

$$
V_1(t-\Delta) \geq \dfrac{1}{2}V_1(t) > C_1
$$

\noindent so $t-\Delta > 1$ and

\be V_1(t-\Delta) = \dfrac{1}{2}V_1(t) \label{E5.3} \ee

\noindent follows.  Define

$$X_C(s) = X(t) + s-t.$$

\noindent Using Theorem 1.2 we have

$$
\begin{array}{rcl}
\left| \dfrac{d}{ds} \left(X_C(s)-X(s)\right)\right|& =&
1-\hat{V}_1(s)\\
\\
& = & \dfrac{(m^{\alpha})^2
+V^2_2(s)}{\sqrt{(m^{\alpha})^2+|V(s)|^2}\left(\sqrt{(m^{\alpha})^2+|V(s)|^2} +
V_1(s)\right)}\\
\\
& \leq & \dfrac{C+C\left(s-|X(s)|+C_0\right)}{V^2_1(s)}.
\end{array}
$$

\noindent Since $s-|X(s)|$ is increasing, for $t-\Delta \leq s \leq t $ we
have

$$
\left|\dfrac{d}{ds} \left(X_C(s) - X(s)\right)\right| \leq
\dfrac{C+C\left(t-|X(t)|+C_0\right)}{\left(\dfrac{1}{2}V_1(t)\right)^2}
$$

\noindent and hence

\be \left|X_C(s) - X(s)\right| \leq \dfrac{C\Delta(t-|X(t)|
+2C_0)}{V^2_1(t)}. \label{E5.4} \ee

By (\ref{E1.3}), the Cauchy Schwartz inequality, and (\ref{E5.1})
we have

$$
\begin{array}{rcl}
\left|\dint_{t-\Delta}^t E_1(s,X(s))ds\right| & = & \left|
\dint^t_{t-\Delta} E_1(s,X_C(s)) ds \right. \\
\\
&& \left.+ \dint^t_{t-\Delta} \dint^{X(s)}_{X_C(s)} \dint
\underset{\alpha}{\dsum} e^{\alpha} f^{\alpha} dv\, dx\,
ds\right|\\
\\
& \leq & C\sqrt{\Delta} + \dint^t_{t-\Delta} \dint^{X(s)}_{X_C(s)}
C \sqrt{k\sigma_-} dx\, ds.
\end{array}
$$

\noindent Now

$$
\dint^t_{t-\Delta} \dint^{X(s)}_{X_C(s)} kdx\,ds \leq C\Delta
$$

\noindent and letting

$$
S(t) = \dfrac{C\Delta(t-|X(t)|+2C_0)}{V^2_1(t)},
$$

\noindent (\ref{E5.4}) and (\ref{E1.3}) yield

$$\begin{array}{rl}
& \dint^t_{t-\Delta} \dint^{X(s)}_{X_C(s)} \sigma_- dx\,ds \leq
\dint^t_{t-\Delta} \dint^{X_C(s)+S(t)}_{X_C(s)} \sigma_- dx\,ds\\
\\
= & \dint^t_{t-\Delta} \dint^{X(t)-t+S(t)}_{X(t)-t}
\sigma_-(s,y+s) dy\,ds\\
\\
= & \dint^{X(t)-t+S(t)}_{X(t)-t} \dint^t_{t-\Delta} \sigma_-
(s,y+s)ds\, dy\\
\\
\leq & CS(t).
\end{array}
$$

\noindent Hence the Cauchy Schwartz inequality yields

$$
\dint^t_{t-\Delta} \dint^{X(s)}_{X_C(s)} C \sqrt{k\sigma_-}\,
dx\,ds \leq C\sqrt{\Delta S(t)}
$$

\noindent and hence

\be \left|\dint^t_{t-\Delta} E_1(s,X(s))ds \right| \leq
C\sqrt{\Delta} + C\sqrt{\Delta S(t)}. \label{E5.5} \ee

Next consider

$$
\dint^t_{t-\Delta} \hat{V}_2 (s) B(s,X(s)) ds.
$$

\noindent Using Theorem 1.2 we have, for $t-\Delta \leq s \leq t$,

$$
\begin{array}{rcl}
\left|\hat{V}_2(s)\right| & \leq & C\dfrac{|V_2(s)|}{V_1(s)} \leq
\dfrac{C+C\sqrt{s-|X(s)| + C_0}}{V_1(s)} \\
\\
& \leq & \dfrac{C\sqrt{t-|X(t)| + 2C_0}}{\frac{1}{2}V_1(t)}.
\end{array}
$$

\noindent Hence by (\ref{E2.5})

\be \dint^t_{t-\Delta} \left|\hat{V}_2 (s) B(s,X(s))\right|\,ds
\leq \dfrac{C\Delta \sqrt{t-|X(t)|+2C_0}}{V_1(t)} . \label{E5.6}
\ee

Collecting (\ref{E5.5}) and (\ref{E5.6}) yields

$$
\begin{array}{rcl}
V_1(t) & = & \left.V_1(t-\Delta) + \dint^t_{t-\Delta} e^{\alpha}
\left(E_1 + \hat{V}_2 (s)B\right)\right|_{(s,X(s))} \,ds\\
\\
& \leq & V_1(t-\Delta) + C\sqrt{\Delta} + C\dfrac{\Delta
\sqrt{t-|X(t)|+2C_0}}{V_1(t)}
\end{array}
$$

\noindent and with (\ref{E5.3}) this becomes

$$
V_1(t) \leq C\sqrt{\Delta} +
C\dfrac{\Delta\sqrt{t-|X(t)|+2C_0}}{V_1(t)}.
$$

\noindent Hence

$$
\begin{array}{rcl}
V^2_1 (t) - C\sqrt{\Delta} V_1(t) & \leq &
C\Delta\sqrt{t-|X(t)|+2C_0},\\
\\
\left(V_1(t)-\dfrac{C\sqrt{\Delta}}{2}\right)^2 & \leq & \Delta
\left(C\sqrt{t-|X(t)|+2C_0} + \dfrac{C^2}{4}\right),
\end{array}
$$

\noindent and

$$\begin{array}{rcl}V_1 (t) &\leq &\dfrac{C\sqrt{\Delta}}{2} + \sqrt{\Delta
\left(C\sqrt{t-|X(t)|+2C_0} + \dfrac{C^2}{4}\right)}\\
\\
& \leq &
C\sqrt{\Delta}\left(1+(t-|X(t)|+2C_0)^{\frac{1}{4}}\right)\\
\\
& \leq & Ct^{\frac{1}{2}}(t-|X(t)|+2C_0)^{\frac{1}{4}}.
\end{array}
$$

Similar estimates may be derived if $V_1(t) < -2C_1$ so

$$
|V_1(t)| \leq 2C_1 + C
t^{\frac{1}{2}}(t-|X(t)|+2C_0)^{\frac{1}{4}}
$$

\noindent in all cases.  Theorem 1.4 follows.

\section{Nonexistence of Steady States}

Consider the monocharge case first.  The dilation identity is

$$
\begin{array}{rl}
& \dfrac{d}{dt} \left( \diint fxv_1 dv\, dx + \dint x
E_2Bdx\right)\\
\\
= & \diint f\left(v_1\hat{v}_1 + x\left(E_1 +\hat{v}_2
B\right)\right)
dv\, dx\\
\\
& + \dint x \left[ \left( -\partial_x B-j_2\right)
B+E_2\left(-\partial_x E_2\right)\right] dx\\
\\
 =& \diint fv_1\hat{v}_1 dv\, dx + \dint x\left(\rho E_1 + j_2 B\right)
dx\\
\\
& - \dint x \left[\partial_x \left(\dfrac{B^2+E^2_2}{2}\right) +
j_2B\right] dx\\
\\
= & \diint  fv_1\hat{v}_1 dv\, dx + \dint x\rho E_1 dx +
\dfrac{1}{2} \dint \left(B^2+E^2_2\right) dx.
\end{array}
$$

\noindent Let

$$M = \diint fdv\,dx$$

\noindent then for $R > C_0 +t$ we have

$$
\dfrac{-M}{2} = E_1(t,-R)\leq E_1(t,x) \leq E_1(t,R) = \dfrac{M}{2}
$$

\noindent for all $x$.  Hence

$$
\begin{array}{rcl}
 \dint x\rho E_1 dx & =& \dfrac{1}{2} \dint^R_{-R} x \partial_x
 E^2_1 dx \\
 \\
 & = & \dfrac{1}{2} \left(R\left(\dfrac{M}{2}\right)^2 -
 \left(-R\right) \left(\dfrac{M}{2}\right)^2 - \dint^R_{-R} E^2_1
 dx \right) \\
 \\
 & = & \dfrac{1}{2} \dint^R_{-R} \left(
 \left(\dfrac{M}{2}\right)^2 - E^2_1 \right) dx \geq 0.
 \end{array}
 $$

 \noindent Hence, for $f$ not identically zero,

 $$
 \dfrac{d}{dt} \left( \diint f x v_1 dv\, dx + \dint x E_2
 Bdx\right) \geq \diint fv_1 \hat{v}_1 dv\, dx > 0
 $$

 \noindent and $f$ cannot be a steady solution.

 Next consider a steady solution in the neutral case.  Note that
 from (\ref{E1.1}) we have $\partial_x E_2 = 0$ so $E_2 =0$ for
 all $x$ follows.  Next note that

 $$
\begin{array}{rl}
& \dfrac{d}{dx} \left( \dint \underset{\alpha}{\dsum}
f^{\alpha}v_1
\hat{v}^{\alpha}_1 dv - \dfrac{1}{2} E^2_1 +\dfrac{1}{2}B^2\right)\\
\\
= & \dint v_1 \underset{\alpha}{\dsum} \hat{v}^{\alpha}_1
\partial_x f^{\alpha} dv - \rho E_1 - j_2B\\
\\
= & - \dint v_1 \underset{\alpha}{\dsum} e^{\alpha}\left[
\left(E_1+\hat{v}^{\alpha}_2 B\right) \partial_{v_1}f^{\alpha} +
\left(E_2 - \hat{v}^{\alpha}_1 B\right)
\partial_{v_2}f^{\alpha}\right] dv\\
\\
& - \rho E_1 - j_2B\\
\\
= & \dint \underset{\alpha}{\dsum} e^{\alpha}f^{\alpha}\left(E_1+
\hat{v}^{\alpha}_2 B\right) dv - \rho E_1 - j_2 B = 0,
\end{array}
$$

\noindent and hence

\be 2\dint \underset{\alpha}{\dsum} f^{\alpha}v_1
\hat{v}^{\alpha}_1 dv = E^2_1 - B^2 \label{E6.1} \ee

\noindent for all $x$.  If $E_1(x) = 0$ for some $x$ then, since $f^{\alpha}\geq
0$,

\be \dint \underset{\alpha}{\dsum} f^{\alpha} v_1
\hat{v}^{\alpha}_1 dv = 0 \label{E6.2} \ee

\noindent follows and then $B(x)=0$ and $f^{\alpha}(x,v)=0$ for
all $v$.  Suppose $E_1(x_0) \neq 0$ for some $x_0$.  A
contradiction will be derived from this and the proof will be
complete.

Choose $a < x_0$ and $b > x_0$ such that

$$
E_1(x) \neq 0\ {\rm on}\ (a,b)
$$

\noindent and

$$
E_1(a) = E_1(b) = 0.$$

\noindent Consider $E_1(x) >0$ on $(a,b)$.  Choose $d \in (a,b)$ with

$$0< E^{\prime}_1 (d) =\dint \underset{\alpha}{\dsum} e^{\alpha}f^{\alpha}
(d,v) dv. $$

\noindent Choose  $\alpha \in \{1, \ldots , N\}$ and $w \in \BR^2$ such that

$$
f^{\alpha}(d,w) > 0
$$
and $e^{\alpha}>0.$  By continuity we may take $w_1 \neq 0$.  Let $(X(s),V(s)) =
\left(X^{\alpha}(s,0,d,w),V^{\alpha}(s,0,d,w)\right)$.  If $w_1 >
0$ define

$$
T = \sup \left\{ t > 0 : V_1(s) \geq 0\ {\rm and}\ X(s) \leq b\
{\rm for\ all}\ s \in [0,t]\right\}.
$$

\noindent On $[0,T), X(s) \in [a,b]$ so $E_1(X(s)) \geq 0$.  From
(\ref{E6.1}) it follows that

$$
\left|B(X(s))\right| \leq E_1 (X(s))
$$

\noindent and hence that

$$\dot{V}_1(s) = e^{\alpha}(E_1(X(s)) + \hat{V}_2 (s) B(X(s))) \geq 0
$$

\noindent and

$$V_1(s) \geq w_1 > 0.
$$

\noindent It follows that $T$ is finite and that

$$
X(T) = b.
$$

\noindent Hence

$$f^{\alpha}(b,V(T)) = f^{\alpha}(d,w) > 0
$$

\noindent which contradicts (\ref{E6.2}).  If $w_1 < 0$ define

$$
T= \inf \left\{ t < 0: V_1(s)\leq 0\ {\rm and}\ X(s) \leq b\ {\rm
for\ all}\ s \in [t,0]\right\}.
$$

\noindent It may be shown that $T$ is finite and that $X(T) = b$, which
again contradicts (\ref{E6.2}).

A contradiction may be reached in a similar manner if $E_1 < 0$ on
$(a,b)$ so the proof is complete.

\end{document}